\documentclass[12pt, reqno]{amsart}
\usepackage{mathrsfs}
\usepackage{amsfonts}
\usepackage{cite}
\usepackage{amsmath}
\usepackage{amsthm}
\usepackage{amssymb}
\usepackage{bm}
\usepackage{bbding}
\usepackage{cases}
\numberwithin{equation}{section}
\newtheorem{thm}{Theorem}[section]
\newtheorem{maintheo}{Theorem}
\newtheorem{lemma}{Lemma}[section]

\newtheorem{cor}{Corollary}[section]

\newtheorem{pro}{Proposition}[section]
\newtheorem{defi}{Definition}[section]

\newtheorem{problem} {Problem}

\usepackage[%
pdfstartview=FitH, %
colorlinks=true, %
linkcolor=blue, %
citecolor=red]{hyperref}

\begin{document}
\title[The Busemann-Petty problem on entropy of log-concave functions]{The Busemann-Petty problem on entropy of log-concave functions}

\author{Niufa Fang}
\address{\parbox[l]{1\textwidth}{School of Mathematics, Hunan University, Changsha, Hunan 300071, China.}}
\email{fangniufa@hnu.edu.cn}

\author{Jiazu Zhou*}
\address{\parbox[l]{1\textwidth}{School of Mathematics and
Statistics, Southwest University, Chongqing 400715, China.\\
College of Science, Wuhan University of Science and Technology, Wuhan, Hubei 430081, China}}
\email{zhoujz@swu.edu.cn }



\thanks{The first author is supported in part by NSFC (No.12001291).}
\thanks{*The corresponding author is supported in part by NSFC (No.12071318).}

\maketitle

\begin{abstract}
The Busemann-Petty problem asks whether symmetric convex bodies in the
 Euclidean space $\mathbb{R}^n$ with smaller central hyperplane sections necessarily have smaller volume.  The solution has  been completed and the answer is affirmative if $n \le 4$ and negative if $n\ge 5$.
In this paper, we investigate the Busemann-Petty problem on entropy of log-concave functions: For even log-concave functions $f$ and $g$  with finite positive integrals in $\mathbb{R}^n$, 
if  the marginal $\int_{\mathbb{R}^n\cap H}f(x)dx$
of $f$ is smaller than   the marginal  $\int_{\mathbb{R}^n\cap H}g(x)dx$
of $g$ for every hyperplane $H$ passing through the origin, whether  the entropy  ${\rm Ent}(f)$    of $f$ is bigger  than the entropy  ${\rm Ent}(g)$ of $g$?
The Busemann-Petty problem on entropy of log-concave functions includes  the   Busemann-Petty problem, hence, its
 answer is  negative  when $n\geq5$.
For $2\leq n\leq4$ we give a positive answer to the Busemann-Petty problem on entropy of log-concave functions.
\end{abstract}

\section{Introduction and main results}
\vskip 0.3cm

The \emph{dual Brunn-Minkowski theory}  introduced  by Lutwak \cite{lutwak1975}  is  a remarkable milestone in convex geometry.
It  has attracted  extensive attention since the \emph{intersection body}
helped achieving a major breakthrough in the solution of the celebrated \emph{Busemann-Petty problem}.
Very recently,   Huang, Lutwak, Yang and Zhang \cite{HuangLYZ} studied the Minkowski problem in the dual Brunn-Minkowski theory,
now known as the \emph{dual Minkowski problem}.
Since their  outstanding  work,  the dual Brunn-Minkowski theory has gained renewed attention.

\vskip 0.1cm

The dual  Burnn-Minkowski theory is an extension of the   Burnn-Minkowski theory  which originated with the work of  Minkowski.  During  the last two decades, many breakthrough works have indeed shown that it is possible to attack important problems in the   Burnn-Minkowski theory  by embedding the class of convex bodies  into appropriate classes of functions or measures on $\mathbb{R}^n$.    There is a natural embedding between the class of log-concave functions   and the class of    convex bodies on $\mathbb{R}^n$,   hence, many concepts and problems of the   Brunn–Minkowski theory have been discovered and studied on the class of log-concave functions (which is now called the geometry of log-concave functions).  See \cite{Alonso1,Alonso,AM,A1,A2,C4,CW,C1,C2,fangzhou1,FXZZ,F1,F2,Lin,Milman2,Rotem2} for more detailed  references.  However,   the ``duals"  of many concepts and problems within the geometry of log-concave functions are widely open. The main  purpose
of this paper is to  study the Busemann-Petty problem for  log-concave functions. 
The functional intersection body  will be defined  which helps
 to solve the Busemann-Petty problems for log-concave functions.

\vskip 0.2cm

In the following, we first recall  Busemann-Petty problem and its history. 
 A subset in the $n$-dimensional   Euclidean space
$\mathbb{R}^n$  is called a \emph{convex body} if it is a compact
convex set with non-empty interior. A convex body $K$ is
origin-symmetric if $K=-K$, where $-K=\{-x: x\in K\}$.  Let
$V_{n-1}(\cdot)$ and $V(\cdot)$ denote the $(n-1)$-dimensional and the $n$-dimensional
Hasusdorff measures, respectively.

\vskip 0.2cm

In 1956, Busemann and Petty  posed the following question \cite{BP}:

{\bf Busemann-Petty problem.}
{\it Suppose that $K$ and $L$ are origin-symmetric convex bodies in the $n$-dimensional Euclidean space $\mathbb R^n$
such that
$$
V_{n-1}(K\cap H)\le V_{n-1}(L\cap H)
$$
for every hyperplane $H$ passing through the origin. Does it follow
that
$$
V(K)\le V(L)?
$$}

 The  Busemann-Petty problem   has a long and dramatic history. It is not so difficult
 to see  that the Busemann-Petty problem has a  positive answer for $n=2$. 
  A negative answer
to the problem for $n\ge 5$ was established in a series of papers by
Larman and Rogers \cite{LG} (for $n\ge 12$), Ball  \cite{Ball1986} ($n\ge 10$), Giannopoulos \cite{giann}
 and Bourgain  \cite{bour} (independently; $n\ge 7$), Gardner \cite{gard1}  and
Papadimitrakis  \cite{papa} (independently; $n\ge 5$).
Intersection bodies introduced by Lutwak \cite{lutwak1988} helped to completely solve the Busemann-Petty problem.    The Busemann-Petty problem can be rephrased in terms of intersection bodies, that is, the Busemann-Petty problem has an affirmative answer in $\mathbb{R}^n$ if and only if every origin-symmetric convex body in $\mathbb{R}^n$ is an intersection body.
  With the help of intersection  bodies, the answer to the Busemann-Petty problem is
affirmative  for  $n \leq 4$.
Gardner showed  an affirmative answer to the Busemann-Petty problem for $n=3$ in \cite{gard2}.
In 1999  Zhang \cite{Zhang2} provided an affirmative answer to the Busemann-Petty problem for $n=4$,
the last unsolved case of
the Busemann-Petty problem.
A  unified solution to
the Busemann-Petty problem for all cases of $n$ was provided  by Gardner, Koldobsky and Schlumprecht in \cite{GKS}.

In the past years,  there are more generalizations of the Busemann-Petty problem, see, for example,
\cite{BZ,Koldobsky4,KKZ,KPZ,KZ,RubinZhang,Wu,Yaskin1,Yaskin2,Zvavitch,Zymonopoulou}.
Moreover, intersection bodies  have  received more  attention and one can infer \cite{Koldobsky1,Koldobsky2,Koldobsky2000,lutwak1991,Emilman2,Zhang1,Zhang3} for more references.

\vskip 0.2cm
The main aim of this paper is to study the Busemann-Petty problem on \emph{entropy} of log-concave functions,  an important concept that is applied in  physics, information theory,
computer science, convex geometry, differential geometry,  probability theory, analysis and other applied mathematical fields.
A function $f:\mathbb{R}^n\rightarrow [0,\infty)$ is \emph{log-concave} if for any  $x,y\in \mathbb{R}^n$ and $0<t<1$, 
\begin{eqnarray}\label{log-concave function}
f((1-t)x+ty)\geq f^{1-t}(x)f^t(y).
\end{eqnarray}
A typical  example of  log-concave functions is 
the characteristic function $\mathcal{X}_K$ of a convex body $K$ in $\mathbb{R}^n$: 
$$
\mathcal{X}_K(x)=\left\{\begin{array}{ccc}
1  & \text{if}&  x\in K,  \\
0  & \text{if }&  x\not\in  K.
\end{array} \right.
$$
The \emph{total mass functional} $J(f)$ of an integrable function $f:\mathbb{R}^n\to\mathbb{R}$ is defined by
\begin{eqnarray}
J(f)=\int_{\mathbb{R}^n}f(x)dx.
\end{eqnarray}
If  $f$ is an integrable log-concave function with $J(f)>0$,   the 
\emph{entropy} ${\rm Ent}(f)$ of $f$ is defined by
\begin{eqnarray}\label{}
{\rm Ent}(f)=\int_{\mathbb{R}^n}f(x)\log f(x)dx-J(f)\log J(f).
\end{eqnarray}
The entropy of log-concave functions provides multiple connections
  between convex bodies and  log-concave functions. In \cite{A2}, the authors provided  a functional version of the affine isoperimetric inequality for log-concave functions which turned out to be  an inverse form of  logarithmic Sobolev inequality for entropy.   Colesanti and  Fragal\`{a} showed that the entropy is one part of the functional form of Minkowski first inequality
  (see, e.g., \cite[Theorem 5.1]{Colesanti}).   Inequalities on entropy were obtained in \cite{C4} by using a  geometric inequality involving the $L_p$ affine surface area.

\vskip 0.2cm
In the present paper, we will study the Busemann-Petty problem for  entropy of log-concave functions   which asks:

\vskip 0.2cm

\begin{problem}\label{entropy case}
 Suppose that $f$ and $ g$ are even  log-concave functions  with finite positive integrals in $\mathbb{R}^n$ such that
\begin{eqnarray*}
\int_{\mathbb{R}^n\cap H}f(x)dx\leq\int_{\mathbb{R}^n\cap H}g(x)dx
\end{eqnarray*}
for each hyperplane $H$ passing through the origin.
Does it follow that
\begin{eqnarray*}
{\rm Ent}(f)
\geq {\rm Ent}(g)?
\end{eqnarray*}
\end{problem}


At first glance, the solution to Problem \ref{entropy case} is  more difficult  than the traditional Busemann-Petty problem since
there is no 
any information in $\mathbb{R}^2$. Similar to the traditional Busemann-Petty problem (see, e.g., \cite{Busemann}), the symmetric assumption on functions is necessary.

\vskip 0.2cm

The work  of Colesanti and  Fragal\`{a} \cite[Theorem 5.1]{Colesanti} tells us  that the entropy is  an  important part in the functional version of Minkowski's first inequality  and another part is the   total mass functional.   It is clear  that the total mass functional is a more direct extension  of volume.  
We consider
 the following  Busemann-Petty problem for the   total mass functional of log-concave functions:

\begin{problem}\label{integral case}
Suppose $f$, $g$ are even  log-concave functions with   finite positive integrals in $\mathbb{R}^n$
such that 
\begin{eqnarray*}
\int_{\mathbb{R}^n\cap H}f(x)dx\leq\int_{\mathbb{R}^n\cap H}g(x)dx
\end{eqnarray*}
for every hyperplane $H$ passing through the origin. Does it follow that
\begin{eqnarray*}
  \int_{\mathbb{R}^n}f(x)dx     
\leq               
\int_{\mathbb{R}^n} g(x)dx ?
\end{eqnarray*}
\end{problem}

When $f$ and $g$ are, respectively,  the characteristic functions of  origin-symmetric convex bodies $K$ and $L$ in $\mathbb{R}^n$, then
 Problem \ref{integral case}  deduces the    Busemann-Petty problem.
 Hence Problem \ref{integral case} has a negative answer for $n\geq 5$.
 Note that  scalar multiplication does not affect  Problem \ref{integral case} 
 and  the function  $t\log  t$ is increasing when $t>e^{-1}$, the crucial step to solve Problem \ref{entropy case} is the following problem:


\vskip 0.3cm
\noindent\textbf{Problem  1$^\prime$.} \emph{If $f, g$ are even  log-concave functions in $\mathbb{R}^n$ with  finite positive integrals such that
\begin{eqnarray*}
\int_{\mathbb{R}^n\cap H}f(x)dx\leq\int_{\mathbb{R}^n\cap H}g(x)dx
\end{eqnarray*}
for every hyperplane $H$ passing through the origin.
Does it follow that
\begin{eqnarray*}
\int_{\mathbb{R}^n}f\log fdx
\geq \int_{\mathbb{R}^n}g\log gdx?
\end{eqnarray*}}
\vskip 0.3cm

%

%

 We realize   that the dual Brunn-Minkowski theory is essential in solving the Busemann-Petty problem. This  inspired us that the dual Brunn-Minkowski theory for functions will be needed in solving  Problems  \ref{entropy case} and \ref{integral case}. Therefore, we extend the harmonic combination of star bodies to functions in Section \ref{Harmonic combination}.

  For any two functions $\varphi,\psi$ (not necessary convex) in $\mathbb{R}^n$,  the \emph{harmonic combination}, $\varphi\widehat{\square} (t\psi): \mathbb{R}^n\to\mathbb{R}$, of $\varphi$ and $\psi$ is defined by
\begin{eqnarray}\label{functional harmonic combination}
\varphi\widehat{\square}(t\psi)=\varphi+t \psi, \quad t>0.
\end{eqnarray}
 In particular, if $\varphi,\psi$ are both  lower semi-continuous convex functions, then
\begin{eqnarray}\label{relation of infimal and harmonic}
\varphi\widehat{\square}(t\psi) =(\varphi^*\square(\psi^*t))^*  ,
\end{eqnarray}
where $\square$ is the infimal convolution and $\varphi^*$ is the Legendre transform (see details  in Section \ref{Preliminaries}). In the sense of (\ref{relation of infimal and harmonic}),  $\varphi\widehat{\square}(t\psi)$ is an extension of the harmonic combination of star bodies.

 For any    two functions  $\varphi,\psi:\mathbb{R}^n \to\mathbb{R}$, let $f=e^{-\varphi}$ and
 $g=e^{-\psi}$. If $f,\ g$ are integrable,
 then   the \emph{functional dual mixed volume} of $f$ and $g$  is defined as

\begin{eqnarray}\label{}
\widehat{\delta J}(f,g)=-\lim_{t\rightarrow 0^{+}}\frac{J(e^{-\varphi\widehat{\square}t\psi(x)})-J(f)}{t}.
\end{eqnarray}
We prove that the dual mixed volume of star bodies is contained in $\widehat{\delta J}(f,g)$.
Moreover, we provide an integral formula of  $\widehat{\delta J}(f,g)$ and  
establish a functional  version of  the dual Minkowski inequality:

\begin{eqnarray}\label{functional-dual-Minkowski-inequality}
\widehat{\delta J}(f,g)\geq \widehat{\delta J}(f,f)+J(f)\log \frac{J(f)}{J(g)},
\end{eqnarray}
with equality if and only if $f=ag$ with $a>0$.

Note that $\widehat{\delta J}(f,f)$ does not agree with $J(f)$, but it satisfies
\begin{eqnarray}
\widehat{\delta J}(f,f)=-\int_{\mathbb{R}^n}f\log fdx.
\end{eqnarray}
Hence, Problem $1^{\prime}$ can be rewritten as

\begin{problem}\label{dual case}
If $f,\  g$ are even  log-concave functions in $\mathbb{R}^n$ with  finite positive integrals such that
\begin{eqnarray*}
\int_{\mathbb{R}^n\cap H}f(x)dx\leq\int_{\mathbb{R}^n\cap H}g(x)dx,
\end{eqnarray*}
for every hyperplane $H$ passing through the origin.
Does it follow that
\begin{eqnarray*}
\widehat{\delta J}(f,f)
\leq\widehat{\delta J}(g,g)?
\end{eqnarray*}
\end{problem}

It is not hard to see that   Problem \ref{integral case} and Problem \ref{dual case}   include   the  Busemann-Petty problem.
In Section \ref{Proof of Theorem} we will show that the    Busemann-Petty problem can be
deduced from Problem \ref{entropy case}.
Therefore Problem \ref{entropy case}, Problem  \ref{integral case}  and Problem  \ref{dual case} can be called the \emph{functional Busemann-Petty problems},
 and   they have negative answers for $n\geq 5$.
Hence we just need to consider the  functional Busemann-Petty problems in $\mathbb{R}^n$  for  $2\leq n\leq 4$.
Obviously, if Problem \ref{integral case} and  Problem \ref{dual case} have  positive answers for $2\leq n\leq 4$ then Problem \ref{entropy case} will
have a positive answer for $2\leq n\leq 4$.

\vskip 0.2cm

For our purposes,  the  functional version of intersection bodies  (or the  intersection  function) is needed. The  \emph{intersection  function} $\mathcal{I} f: \mathbb{R}^n\rightarrow [0,\infty)$,  of a non-negative integrable function $f:\mathbb{R}^n\to\mathbb{R}$, is defined as
\begin{eqnarray}\label{intersection-function}
\mathcal{I} f (x)=\exp\left\{-\|x\|\left(\int_{\mathbb{R}^n\cap \bar{x}^{\perp}}f(z)dz\right)^{-1}\right\},
\end{eqnarray}
when  $x\in \mathbb{R}^n\setminus\{0\}$   and $\mathcal{I} f (0)=1$. Here $\|x\|$ denotes  the Euclidean  normal of $x\in \mathbb{R}^n$ and $\bar{x}=\frac{x}{\|x\|}$ for $x\in \mathbb{R}^n\setminus\{0\}$. 

Let $g: \mathbb{R}^n\rightarrow [0,\infty)$. 
If there exists a non-negative  integrable function $f$  such that $g(x)=\mathcal{I} f (x)$ then $g$ is an {\it intersection function}.
Let $\|f\|_{\infty}^{}=\sup_{x\in \mathbb{R}^n}|f(x)|$, and $\omega_n=\pi^{\tfrac{n}{2}}/\Gamma(1+\tfrac{n}{2})$ denote the volume of the unit ball, where $\Gamma(\cdot)$ is the Gamma function.  We obtain  the following basic inequality for intersection functions:

\vskip 0.2cm

\textbf{The functional Busemann intersection inequality.}
\emph{Let $f: \mathbb{R}^n \to \mathbb{R}$ be a non-negative continuous integrable function
with $f(0)>0$. Then
\begin{eqnarray}\label{functional-Busemann-intersection-inequality}
J(\mathcal{I} f)\leq\Gamma(n+1)\frac{\omega_{n-1}^{n}}{\omega_{n}^{n-2}}
\|f\|_{\infty}J(f)^{n-1},
\end{eqnarray}
with  equality   for $n=2$ if and only if $cf$ is the  characteristic function of  an  origin-symmetric star body, and for $n\geq 3$ if and only if $cf$ is  the characteristic function of  an  origin-symmetric ellipsoid (where $c>0$).}

\vskip 0.2cm

The intersection function not only contains the intersection body (of star bodies) but also plays a crucial role in 
solving the functional 
Busemann-Petty problems.

\vskip 0.2cm
\textbf{Characteristic theorem.} \emph{Problem \ref{integral case} and Problem \ref{dual case} have  affirmative answers in $\mathbb{R}^n$ if and only if every even integrable log-concave function in $\mathbb{R}^n$ is an intersection function.}

\vskip 0.2cm

For a given continuous function (not necessary log-concave), Ball \cite{Ball1988} introduced a body  associated with it. Using Ball's body, a further  connection between the  intersection  function and the intersection body will be given in Lemma \ref{relation of intersection function and body}.  
Together with the solution to the  Busemann-Petty problem assure that Problem \ref{integral case} and Problem \ref{dual case}  have   affirmative answers when $2\leq n\leq 4$. Therefore, we conclude that:

\begin{maintheo}\label{main theorem}
Functional Busemann-Petty problems (i.e., Problem \ref{entropy case}, Problem \ref{integral case} and Problem \ref{dual case}) have  affirmative answers  when $2\leq n\leq 4$ and     negative  answers  when $n\geq 5$.
\end{maintheo}

This paper is organized as follows. In Section \ref{Preliminaries}, we recall some basic facts about convex bodies, log-concave functions and Rodan transform. In Section \ref{Harmonic combination}, the harmonic combination of functions and its related basic inequality   are developed. In Section \ref{intersection function}, we introduce the intersection function and the  functional Busemann intersection inequality is obtained. In Section \ref{FBP}, we give positive answers to Problem \ref{integral case} and \ref{dual case}. In Section \ref{Proof of Theorem}, we complete the proof of Theorem \ref{main theorem}.


\section{Preliminaries}\label{Preliminaries}
\vskip 0.3cm

\subsection{Harmonic combination}
~~\vskip 0.3cm

Let $S^{n-1}=\{x\in \mathbb{R}^n: \|x\|= 1\}$ denote the unit sphere in
$\mathbb{R}^n$, and let $V(K)$ denote the $n$-dimensional volume of convex  body (i.e., a compact,
convex subset with nonempty interior) $K$. Let $\mathcal{K}_o^n$ denote the set of convex bodies containing the origin in their interiors.  We write $\text{GL}(n)$ for the group of general
linear transformations in $\mathbb{R}^n$. For $T\in\text{GL}(n)$ write $T^t$ for the transpose of $T$,  $T^{-t}$ for the inverse of the
transpose (contragradient) of $T$.

For a convex body $K$ in $\mathbb{R}^n$, its support function,  $h(K,\cdot)=h_K(\cdot):\mathbb{R}^n\rightarrow\mathbb{R}$, is defined by
\begin{eqnarray}
h(K, x) = \max\{x\cdot y: y \in K \}
\end{eqnarray}
for $x\in\mathbb{R}^n$, where $x\cdot y$ is the usual inner product of $x$ and $y$. The \emph{polar body},  $K^{\circ}$, of $K$  is defined  by
\begin{eqnarray}
K^{\circ}=\left\{x\in\mathbb{R}^n: x\cdot y \leq 1\quad\text{for all}\quad y\in K  \right\}.
\end{eqnarray}

The radial function, $\rho_K(\cdot)=\rho(K,\cdot):\mathbb{R}^n\setminus\{0\}\rightarrow [0,\infty)$, of a compact, star shaped (about the origin) $K\subset \mathbb{R}^n$, is defined, for $x\neq 0$, by
\begin{eqnarray}
\rho(K,x)=\max\{t\geq 0: t x\in K\}.
\end{eqnarray}
If $\rho_K$ is positive and continuous, call $K$ a \emph{star body}  (about the origin). Let $\mathcal{S}_o^n$ denote the set of star bodies (about the origin) in $\mathbb{R}^n$. Two star bodies $K$ and $L$ are said to dilates (of one another) if $\frac{\rho_K(u)}{\rho_L(u)}$ is independent of $u\in S^{n-1}$.
The radial function is positively homogeneous of degree $-1$, that is,
\begin{eqnarray}\label{positively homogeneous}
\rho_K(tx)=t^{-1}\rho_K(x),\quad t>0.
\end{eqnarray}
For star bodies $K,L\in \mathcal{S}_o^n$ and real numbers $s,t>0$ and $p\geq 1$, the harmonic $p$-combination $s\diamond K\widehat{+}_p t\diamond L$ is defined as

\begin{eqnarray}\label{}
\rho(s\diamond K\widehat{+}_p t\diamond L,\cdot)^{-p}=s\rho( K,\cdot)^{-p}+t\rho(L,\cdot)^{-p}.
\end{eqnarray}
In particular,
\begin{eqnarray*}\label{}
s\diamond K=s^{-\frac{1}{p}} K.
\end{eqnarray*}
When $K, L\in \mathcal{K}_o^n$, the harmonic $p$-combination $K\widehat{+}_pL$  can rewrite as
\begin{eqnarray}\label{relationship}
K\widehat{+}_pL=(K^{\circ}+_pL^{\circ})^{\circ},
\end{eqnarray}
where $K+_pL$ denotes the $L_p$ Minkowski sum of $K, L\in\mathcal{K}_o^n$ with support function
\begin{eqnarray*}\label{}
h_{K+_pL}^p=h_K^p+h_L^p.
\end{eqnarray*}

The $L_p$ dual mixed volume of $K,L\in\mathcal{S}_o^n$ is defined as
\begin{eqnarray}\label{p-dual mixed volume}
\widetilde{V}_{-p}(K,L)=-\frac{p}{n}\lim_{t\rightarrow 0^{+}}\frac{V(K\widehat{+}_p t\diamond L)-V(K)}{t}.
\end{eqnarray}
The definition above and the polar coordinate formula for volume give the following integral representation of the $L_p$ dual mixed volume:
\begin{eqnarray}\label{}
\widetilde{V}_{-p}(K,L)=\frac{1}{n}\int_{S^{n-1}}\rho_K^{n+p}(u)\rho_L^{-p}(u)du.
\end{eqnarray}
The basic inequality for the dual mixed volumes $\widetilde{V}_{-p}$ is that for star bodies $K,L\in\mathcal{S}_o^n$,
\begin{eqnarray}\label{dual minkowski inequality}
\widetilde{V}_{-p}(K,L)^n\geq V(K)^{n+p}V(L)^{-p},
\end{eqnarray}
with equality if and only if $K$ and $L$ are dilates.

For a star body $K\in \mathcal{S}_o^n$, we denote its   \emph{Minkowski functional}  by
\begin{eqnarray*}\label{}
\|x\|_K=
\left\{\begin{array}{ccc}
\rho_K^{-1}(x)  & \text{if}&  x\neq 0,  \\
0  &  \   \text{if }&  x=0.
\end{array} \right.
\end{eqnarray*}

\vskip 0.3cm
~~
\subsection{Log-concave functions}
~~~\vskip 0.3cm

Let $\varphi: \mathbb{R}^n\rightarrow \mathbb{R}\cup \{+\infty\}$. Then  $\varphi$ is convex  if for every $x,y\in\mathbb{R}^n$ and $\lambda\in[0,1]$
$$
\varphi((1-\lambda)x+\lambda y)\leq (1-\lambda)\varphi(x)+\lambda \varphi(y).
$$
Let
$$
\text{dom}(\varphi)=\{x\in \mathbb{R}^n: \varphi(x)\in \mathbb{R}\}.
$$
We say that $\varphi$ is \emph{proper} if $\text{dom}(\varphi)\neq\emptyset$.
   The \emph{Legendre transform} of $\varphi$ is the convex function defined by
\begin{eqnarray}\label{Fenchel conjugate}
\varphi^*(y)=\sup_{x\in\mathbb{R}^n}\left\{ x\cdot y-\varphi(x)\right\}\quad\quad\forall y\in\mathbb{R}^n.
\end{eqnarray}
For convex functions $\varphi, \psi$, the \emph{infimal convolution} is defined by
\begin{eqnarray}\label{infimal convolution}
\varphi\Box \psi(x)=\inf_{y\in\mathbb{R}^n}\{\varphi(x-y)+\psi(y)\} \quad \forall x\in\mathbb{R}^n,
\end{eqnarray}
and the \emph{right scalar multiplication}  is defined by,
\begin{eqnarray}\label{right scalar multiplication}
(\varphi t)(x)=t \varphi\left(\frac{x}{t}\right), \quad\text{for}\quad t>0.
\end{eqnarray}

From the definition of log-concave function (\ref{log-concave function}), we known that every  log-concave function $f:\mathbb{R}^n\rightarrow \mathbb{R}$  has  the form
\begin{eqnarray*}
 f=e^{-\varphi},
\end{eqnarray*}
where $\varphi:\mathbb{R}^n\rightarrow \mathbb{R}\cup \{+\infty\}$ is convex. A log-concave function is degenerate if it  vanishes almost everywhere in $\mathbb{R}^n$.  A non-degenerate, log-concave  function $e^{-\varphi}$ is integrable on $\mathbb{R}^n$ if and only if \cite{Cordero-ErausquinKlartag}
$$
\lim_{\|x\|\rightarrow+\infty}\varphi(x)=+\infty.
$$
Note that all log-concave functions are differentiable almost everywhere in $\mathbb{R}^n$.

Let
\begin{eqnarray*}\label{}
\mathcal{L}&=&\left\{\varphi:\mathbb{R}^n\rightarrow \mathbb{R}\cup \{+\infty\} \Big| \quad  \varphi \text{ is proper, convex, }\lim_{\|x\|\rightarrow+\infty}\varphi(x)=+\infty\right\},\\
\mathcal{A}&=&\left\{f:\mathbb{R}^n\rightarrow \mathbb{R} \Big| \quad  f=e^{-\varphi},\varphi\in \mathcal{L}\right\}.
\end{eqnarray*}
The subset  of $\mathcal{A}$ that contains the even log-concave functions is denoted by $\mathcal{A}_e$.

\vskip 0.3cm
~~
\subsection{The Rodan transform}
~~~\vskip 0.3cm


Let $f$ be a compactly supported continuous function in  $\mathbb{R}^n$ and $\Sigma_n$ be  the space of hyperplanes in $\mathbb{R}^n$.
The \emph{Radon transform} $\mathcal{R} f$ of $f$  is defined by  \cite{He}

\begin{eqnarray*}
\mathcal{R}f(\xi)=\int_{\Sigma_n}f(x)dx, \quad \xi \in \Sigma_n.
\end{eqnarray*}
Here the integral is taken with respect to the natural hypersurface measure $dx$. Observe that any element of $\Sigma_n$ is characterized as the solution locus of an equation
\begin{eqnarray*}
y\cdot u=r,
\end{eqnarray*}
where $u\in S^{n-1}$ is a unit vector and $r\in \mathbb{R}$. Thus the $n$-dimensional Radon transform
 be rewritten as a function on $S^{n-1}\times \mathbb{R}$ via
\begin{eqnarray}
\mathcal{R}f(u,r)=\int_{\{y\in \mathbb{R}^n:u\cdot y=r\}}f(y)dy.
\end{eqnarray}
The Rodan transform of $f$ in the direction of $x$ at $r$ is defined by
\begin{eqnarray}\label{radon transform}
\mathcal{R}f(x,r)=\|x\|^{-1}\mathcal{R}f\left(\frac{x}{\|x\|},\frac{r}{\|x\|}\right), \quad x\in \mathbb{R}^n\setminus\{0\}.
\end{eqnarray}

For a continuous function $f$ on $S^{n-1}$, the \emph{spherical Rodan transform} $R f$ of $f$ is defined by
\begin{eqnarray*}
{\rm R} f(u)=\int_{S^{n-1}\cap u^{\bot}}f(v)dv, \quad\quad u\in S^{n-1},
\end{eqnarray*}
where $u^{\bot}$ is the $(n-1)-$ dimensional subspace orthogonal to the unit vector $u$.

\vskip 0.2cm
~~
\section{The harmonic combination for functions}\label{Harmonic combination}
 Motivated by  (\ref{relationship}), connection between  harmonic $p$-combination and $L_p$ Minkowski sum,   the harmonic combination of convex functions
$\varphi,\psi$   in $\mathbb{R}^n$ is defined as

\begin{eqnarray}\label{}
\varphi\widehat{\square}(t\psi)=(\varphi^*\square(\psi^*t))^* ,\quad t>0.
\end{eqnarray}
 If $\varphi,\psi$ are lower semi-continuous convex functions then (see, e.g., \cite{Colesanti})

\begin{eqnarray*}\label{}
(\varphi^*\square(\psi^*t))^* =\varphi+t\psi, \quad t>0. 
\end{eqnarray*}

We  extend the harmonic combination of convex functions to  non-convex functions as following:

\begin{defi}
For any two functions $\varphi,\psi$ (not necessary convex) in $\mathbb{R}^n$, the harmonic combination $\varphi\widehat{\square} (t\psi)$ of $\varphi$ and $\psi$ is defined by

\begin{eqnarray}\label{}
\varphi\widehat{\square}(t\psi)(x)=\varphi(x)+t \psi(x),
\end{eqnarray}
for $t>0$ and $x\in\mathbb{R}^n$.
\end{defi}


We extend the  dual mixed volume  in  the dual Brunn-Minkowski theory to   its functional version.

\begin{defi}\label{mixed_v(f,g)}
For $\varphi, \psi:\mathbb{R}^n \to \mathbb{R}$, let $f=e^{-\varphi}, g=e^{-\psi}$. If $f, g$ are integrable, then the functional dual mixed volume of $f$ and $g$ is defined as

\begin{eqnarray}\label{}
\widehat{\delta J}(f,g)=-\lim_{t\rightarrow 0^{+}}\frac{J(e^{-\varphi\widehat{\square}t\psi(x)})-J(f)}{t}.
\end{eqnarray}
\end{defi}

 Definition \ref{mixed_v(f,g)} includes the $L_p$ dual mixed volume of star bodies.

\begin{lemma}\label{includes geometry case}
  Let $K, L\in \mathcal{S}_o^n$ and $p\geq 1$. If $\varphi(x)=\|x\|_K^{p}, \psi(x)=\|x\|_L^{p}$ and $f=e^{-\varphi}$ and $g=e^{-\psi}$, then
\begin{eqnarray*}\label{}
J(f)=\Gamma(1+\tfrac{n}{p})V(K),
\end{eqnarray*}
and
\begin{eqnarray*}\label{}
\widehat{\delta J}(f,g)=\frac{n}{p}\Gamma(1+\tfrac{n}{p})\widetilde{V}_{-p}(K,L).
\end{eqnarray*}

\end{lemma}

\begin{proof}
 If $\varphi(x)=\|x\|_K^{p}$ and $ \psi(x)=\|x\|_L^{p}$, a direct calculation shows that
$$
J(f)=\Gamma(1+\tfrac{n}{p})V(K),
$$
and
$$
\varphi(x)+t\psi(x)=\rho_K^{-p}(x)+t\rho_L^{-p}(x)=\rho_{K\hat{+}_pt\diamond L}^{-p}(x)
$$
for $x\in \mathbb{R}^n\setminus\{0\}$. The definition of $L_p$ dual mixed volume (\ref{p-dual mixed volume}) implies
\begin{eqnarray*}\label{}
\widehat{\delta J}(f,g)&=&-\lim_{t\rightarrow 0^{+}}\frac{J(e^{-\rho_{K\hat{+}_pt\diamond L}^{-p}(x)})-J(e^{-\rho_K^{-p}(x)})}{t}\\
&=&\frac{n}{p}\Gamma(1+\tfrac{n}{p})\widetilde{V}_{-p}(K,L).
\end{eqnarray*}
\end{proof}

Next, we give an  integral representation for $\widehat{\delta J}(f,g)$.

\begin{thm}\label{integral formula}
Let  $\varphi,\psi$  be any two functions in $\mathbb{R}^n$ such that  $f=e^{-\varphi},g=e^{-\psi}$ are integrable. If $J(f)>0$, then  $\widehat{\delta J}(f,g)\in [cJ(f),+\infty]$ with $c=\inf \psi$, and
\begin{eqnarray}\label{integral representation}
\widehat{\delta J}(f,g)=\int_{\mathbb{R}^n}\psi(x)e^{-\varphi^{}(x)} dx.
\end{eqnarray}
\end{thm}
\begin{proof}
Set
$$
f_t=e^{-\varphi\widehat{\square} (t\psi)}=e^{-(\varphi+t\psi)},
\quad
c=\inf \psi, \quad \tilde{f}_t=e^{-(\varphi+t(\psi-c))}.
$$
Since $\psi-c\geq0$, then $\tilde{f}_t$ is pointwise decreasing with respect to $t$, and for every $x\in \mathbb{R}^n$ we have $f(x)=\lim_{t\rightarrow0^{+}}\tilde{f}_t(x)$.
By monotone convergence theorem, we have $\lim_{t\rightarrow0^{+}}J(\tilde{f}_t)=J(f)$.  Since $
f_t=e^{-ct}\tilde{f}_t$, we have
\begin{eqnarray}\label{integral representation1}
\frac{J(f_t)-J(f)}{t}
=\frac{e^{-tc}[J(\tilde{f}_t)-J(f)]}{t}+\frac{(e^{-tc}-1)J(f)}{t}.
\end{eqnarray}

Let us consider separately the following two cases:

\begin{eqnarray*}\label{}
\exists \quad t_0>0\quad \text{such that}\quad J(\tilde{f}_{t_0})= J(f),
\end{eqnarray*}
and
\begin{eqnarray*}\label{}
J(\tilde{f}_t)< J(f)\quad \text{for all}\quad t>0.
\end{eqnarray*}
In the first case, since $J(\tilde{f}_t)$ is a decreasing function of $t$, necessarily it holds
$J(\tilde{f}_{t})= J(f)$ for every $t\in[0,t_0]$. Hence the first addendum in the right hand side
of (\ref{integral representation1}) is zero, and  we have
\begin{eqnarray*}\label{}
\lim_{t\rightarrow0^{+}}\frac{J(f_t)-J(f)}{t}
=-cJ(f).
\end{eqnarray*}
For the second case, we write
 \begin{eqnarray*}\label{}
\frac{J(\tilde{f}_t)-J(f)}{t}=\frac{\log J(\tilde{f}_t)-\log J(f)}{t}\cdot\frac{J(\tilde{f}_t)-J(f)}{\log J(\tilde{f}_t)-\log J(f)}.
\end{eqnarray*}
By H\"older inequality 
$\log J(\tilde{f}_t)$ is a decreasing convex function of $t$. Hence

\begin{eqnarray}\label{integral representation2}
\lim_{t\rightarrow0^{+}}\frac{\log J(\tilde{f}_t)-\log J(f)}{t}\in[-\infty,0],
\end{eqnarray}
and
\begin{eqnarray}\label{integral representation3}
\lim_{t\rightarrow0^{+}}\frac{J(\tilde{f}_t)-J(f)}{\log J(\tilde{f}_t)-\log J(f)}= J(f)>0.
\end{eqnarray}
From (\ref{integral representation2}) and (\ref{integral representation3}), we have

\begin{eqnarray}\label{integral representation4}
\lim_{t\rightarrow0^{+}}\frac{J(\tilde{f}_t)-J(f)}{t}\in[-\infty,0].
\end{eqnarray}
Together with  (\ref{integral representation1}) and (\ref{integral representation4}), we have
\begin{eqnarray*}\label{}
\lim_{t\rightarrow0^{+}}\frac{J(f_t)-J(f)}{t}\in[-\infty,-cJ(f)].
\end{eqnarray*}
Therefore, the definition of $\widehat{\delta J}(\cdot,\cdot)$ concludes  that $\widehat{\delta J}(f,g)\in[cJ(f),+\infty]$.

To prove (\ref{integral representation}), firstly we  assume
$\psi\geq 0$. Then
\begin{eqnarray*}\label{}
\frac{J(e^{-\varphi\widehat{\square}(t\psi)(x)})-J(e^{-\varphi})}{t}
&=&\int_{\mathbb{R}^n}e^{-\varphi^{}(x)}\left(\frac{e^{-t\psi^{}(x)}-1}{t}\right)dx.
\end{eqnarray*}
Since $\psi\geq 0$, 
by the monotone convergence theorem we have
\begin{eqnarray*}\label{}
-\lim_{t\rightarrow 0^{+}}\frac{J(e^{-\varphi\widehat{\square}(t\psi)(x)})-J(e^{-\varphi})}{t}
&=&\int_{\mathbb{R}^n}\psi(x)e^{-\varphi^{}(x)}dx.
\end{eqnarray*}
In the general case when the assumption $\psi\geq 0$ is removed, we consider the function $\tilde{f}_t$.  Let $\tilde{\psi}=\psi-c$.  From (\ref{integral representation1}) and the monotone convergence theorem, we have
\begin{eqnarray*}\label{}
\widehat{\delta J}(f,g)
&=&-\lim_{t\rightarrow0^{+}}\frac{e^{-tc}J(e^{-\varphi\widehat{\square} (t\tilde{\psi})})-e^{-tc}J(f)}{t}-\lim_{t\rightarrow0^{+}}\frac{(e^{-tc}-1)J(f)}{t}\\
&=&\int_{\mathbb{R}^n}\tilde{\psi}(x)e^{-\varphi^{}(x)}dx+cJ(f)\\
&=&\int_{\mathbb{R}^n}\psi(x)e^{-\varphi^{}(x)} dx.
\end{eqnarray*}
We complete the  proof of Theorem \ref{integral formula}.
\end{proof}

We are going to  present a functional form of the dual Minkowski
inequality.

\begin{lemma}
Let  $\varphi,\psi$  be any two functions in $\mathbb{R}^n$ such that  $f=e^{-\varphi},g=e^{-\psi}$ are integrable. Then
\begin{eqnarray}\label{}
\lim_{t\rightarrow 0^{+}}\frac{J(e^{-[(1-t)\varphi]\widehat{\square} [t\psi]})-J(e^{-\varphi^{}})}{t}=\widehat{\delta J}(f,f)-\widehat{\delta J}(f,g)
\end{eqnarray}
\end{lemma}

\begin{proof}
For $t\in(0,1)$, let $a(t)=\frac{t}{1-t}$.
 Set $\varphi_{a(t)}=\varphi\widehat{\square}[a(t)\psi]$, and
\begin{eqnarray*}\label{}
b_t(s)=\int_{\mathbb{R}^n}e^{-(1-s)\varphi_{a(t)}(x)}dx, \quad s\in (0,1).
\end{eqnarray*}
Then
\begin{eqnarray*}\label{}
\frac{d}{ds}b_t(s)=\int_{\mathbb{R}^n}e^{-(1-s)\varphi_{a(t)}(x)}\varphi_{a(t)}(x)dx.
\end{eqnarray*}

First, we assume that $\psi\geq0$. Then for every fixed $t\in (0,1)$, by Lagrange theorem 
there exists a $s_0\in (0,1)$ such that
\begin{eqnarray*}\label{}
&&\lim_{t\rightarrow 0^{+}}\frac{J(e^{-[(1-t)\varphi]\widehat{\square}[t\psi](x)})-J(e^{-\varphi})}{t}\\
&&\quad=\lim_{t\rightarrow 0^{+}}\frac{1}{t}\left[\int_{\mathbb{R}^n}e^{-(1-t)[\varphi(x)+a(t)\psi(x)]}dx
-\int_{\mathbb{R}^n}e^{-[\varphi(x)+a(t)\psi(x)]}dx\right]\\
&&\quad\quad +\lim_{t\rightarrow 0^{+}}\frac{1}{t}\left[\int_{\mathbb{R}^n}e^{-[\varphi(x)+a(t)\psi(x)]}dx-\int_{\mathbb{R}^n}e^{-\varphi(x)}dx\right]\\
&&\quad=b_t'(s_0)+\lim_{t\rightarrow 0^{+}}\frac{1}{t}\left[\int_{\mathbb{R}^n}e^{-[\varphi(x)+a(t)\psi(x)]}dx-\int_{\mathbb{R}^n}e^{-\varphi(x)}dx\right].
\end{eqnarray*}
Since $\psi\geq0$, so that $\varphi(x)+a(t)\psi(x)$ is increasing with respect to $t$.
Then, from the monotone convergence theorem, the fact $s_0\rightarrow0^{+}$ and $a(t)\rightarrow0^{+}$ as $t\rightarrow0^{+}$, we have
\begin{eqnarray*}\label{}
&&\lim_{t\rightarrow 0^{+}}\frac{J(e^{-[(1-t)\varphi]\widehat{\square}[t\psi](x)})-J(e^{-\varphi})}{t}\\
&&\quad=\int_{\mathbb{R}^n}e^{-\varphi_{}(x)}\varphi_{}(x)dx-\int_{\mathbb{R}^n}e^{-\varphi_{}(x)}\psi_{}(x)dx.
\end{eqnarray*}
The desired equality follows from the integral representation (\ref{integral representation}).

It remaind to remove the assumption $\psi\geq0$. In the general case, we set
$$
c=\inf \psi, \quad \tilde{\psi}=\psi-c.
$$
Therefore,
\begin{eqnarray*}\label{}
&&\lim_{t\rightarrow 0^{+}}\frac{J(e^{-[(1-t)\varphi]\widehat{\square}[t\psi](x)})-J(e^{-\varphi})}{t}\\
&&\quad=\lim_{t\rightarrow 0^{+}}e^{-tc}\frac{J(e^{-[(1-t)\varphi(x)
  +t(\psi(x)-c)]})-J(e^{-\varphi})}{t}+\lim_{t\rightarrow 0^{+}}J(e^{-\varphi})\frac{e^{-tc}-1}{t}\\
&&\quad=\int_{\mathbb{R}^n}(\varphi^{}(x)-\psi^{}(x))e^{-\varphi^{}(x)}dx.
\end{eqnarray*}
This finishes the proof of the lemma.
\end{proof}

The following theorem gives a functional form of the dual Minkowski inequality.

\begin{thm}\label{functional dual Minkowski inequality0}
Let  $\varphi,\psi$  be any two functions in $\mathbb{R}^n$ such that  $f=e^{-\varphi},g=e^{-\psi}$ are integrable. If $J(f)>0$ and $J(g)>0$, then
\begin{eqnarray}\label{functional dual Minkowski inequality}
\widehat{\delta J}(f,g)\geq \widehat{\delta J}(f,f)+J(f)\log \frac{J(f)}{J(g)},
\end{eqnarray}
with equality if and only if $f=ag$ with $a>0$.
\end{thm}
\begin{proof}
By H\"older  inequality we have
\begin{eqnarray*}\label{}
J(e^{-(1-t)\varphi-t\psi})
\leq J(e^{-\varphi})^{(1-t)}J(e^{-\psi})^t \quad \text{for} \ t\in (0,1),
\end{eqnarray*}
i.e.,
\begin{eqnarray*}\label{}
\log J(e^{-(1-t)\varphi-t\psi})
\leq (1-t)\log J(e^{-\varphi})+t\log J(e^{-\psi}).
\end{eqnarray*}
This means that $\Upsilon(t)=\log J(e^{-(1-t)\varphi-t\psi})$ is a convex  function on $[0,1]$, hence
\begin{eqnarray}\label{main inequality 0}
\Upsilon(t)\leq\Upsilon(0)+t[\Upsilon(1)-\Upsilon(0)],
\end{eqnarray}
for $t\in[0,1]$. As a consequence, the derivative of the function $\Upsilon$ at $t=0$ satisfies
\begin{eqnarray}\label{main inequality1}
\Upsilon'(0)\leq\Upsilon(1)-\Upsilon(0).
\end{eqnarray}
Therefore,
\begin{eqnarray*}\label{}
\frac{\widehat{\delta J}(f,f)-\widehat{\delta J}(f,g)}{J(f)}\leq \log J(g)-\log J(f).
\end{eqnarray*}

Finally, we assume that $f=a g$ for some $a>0$. Then
(\ref{functional dual Minkowski inequality}) holds with equality
sign
 by (\ref{integral representation}). Conversely, assume that (\ref{functional dual Minkowski inequality}) holds with equality sign.
 By inspection of the above proof one sees immediately that inequality in (\ref{main inequality 0}) also holds,
 and hence the equality must hold in inequality (\ref{main inequality1}). 
This entails that the  H\"older  inequality holds as an equality, and therefore there exists a $a>0$ such that $f=ag$.
\end{proof}

For $t>0$, note that
$$
\log \frac{1}{t}\geq 1-t,
$$
with equality if and only if $t=1$. Therefore, by (\ref{functional dual Minkowski inequality}) we have:

\begin{cor}\label{needed inequality}
Let  $\varphi,\psi$  be any two functions in $\mathbb{R}^n$ such that  $f=e^{-\varphi},g=e^{-\psi}$ are integrable. Then
\begin{eqnarray}
\widehat{\delta J}(f,g)\geq \widehat{\delta J}(f,f)+J(f)-J(g),
\end{eqnarray}
with equality if and only if $f=g$.

\end{cor}

The functional form of the dual Minkowski inequality  (\ref{functional dual Minkowski inequality}) includes the following geometric inequality.

\begin{cor}
Let $K, L\in \mathcal{S}_o^n$ and $p\geq 1$. If $\varphi(x)=\|x\|_K^{p}, \psi(x)=\|x\|_L^{p}$,  then (\ref{functional dual Minkowski inequality}) implies
\begin{eqnarray}\label{weak}
\widetilde{V}_{-p}(K,L)\geq V(K)+V(K)\log\left(\frac{V(K)}{V(L)}\right)^{\frac{p}{n}},
\end{eqnarray}
with equality if and only if there exists a constant  $t\in \mathbb{R}$ such that $\rho_K^{-p}(x)=\rho_L^{-p}(x)+t$.
\end{cor}
\begin{proof}
By a direct calculation, we have
\begin{eqnarray*}\label{}
\widehat{\delta J}(f,f)&=&\int_{\mathbb{R}^n}\varphi(x)e^{-\varphi(x)} dx\\
&=&\int_{\mathbb{R}^n\setminus \{0\}}\rho_K^{-p}(x)e^{-\rho_K^{-p}(x)} dx\\
&=&\int_{0}^{\infty}\int_{S^{n-1}}r^{n+p-1}\rho_K^{-p}(u)e^{-r^p\rho_K^{-p}(u)} drdu\\
&=&\frac{1}{p}\int_{0}^{\infty}\int_{S^{n-1}}r^{\frac{n}{p}}\rho_K^{n}(u)e^{-t} dtdu\\
&=&\tfrac{n}{p}\Gamma(\tfrac{n}{p}+1)V(K).
\end{eqnarray*}
 The inequality (\ref{weak}) follows from Lemma \ref{includes geometry case} and Theorem \ref{functional dual Minkowski inequality0}. The equality condition of Theorem \ref{functional dual Minkowski inequality0} implies that  $e^{-\rho_K^{-p}(x)}=ae^{-\rho_L^{-p}(x)}$ with $a>0$, that is, $\rho_K^{-p}(x)=\rho_L^{-p}(x)-\log a$.
\end{proof}

\section{Intersection bodies for functions}\label{intersection function}

\subsection{Marginals}
~~\vskip 0.3cm

In this subsection, we will introduce 
marginals of a function. More detailed  information can be found in the book \cite{BGVV}.

Denote by $G_{n,k}$, the Grassmann manifold of $k$-dimensional linear subspaces in $\mathbb{R}^n$, $1 \leq k \leq n-1$.
 Let  $f:\mathbb{R}^n\rightarrow \mathbb{R}$ be an integrable function.
 Let $1\leq k<n$ be an integer and let $F\in G_{n,k}$.
 The \emph{marginal} $\pi_F(f): F\rightarrow[0,\infty)$ of $f$ with respect to $F$ is defined by

\begin{eqnarray}\label{section function}
\pi_F(f)(x)=\int_{x+F^{\bot}}f(y)dy,
\end{eqnarray}
where $F^{\bot}$ denotes the orthocomplement space of $F$.

We need the following properties of marginals.

 \begin{lemma}\label{BGVV's Proposition}
\emph{ (\cite[Proposition 5.1.11]{BGVV}).}
Let $f:\mathbb{R}^n\rightarrow \mathbb{R}$ be an integrable function and let  $1\leq k<n$  and $F\in G_{n,k}$.
\begin{enumerate}
\item If $f$ is log-concave, then $\pi_F(f)$ is log-concave.
\item We have $\int_F\pi_F(f)(x)=\int_{\mathbb{R}^n}f(x)dx$.
\item If $f$ is even then the same holds true for $\pi_F(f)$.
\end{enumerate}
\end{lemma}

For every $x\in \mathbb{R}^n\setminus\{0\}$ and $\bar{x}=\frac{x}{\|x\|}$, we define the \emph{parallel section function} $t\mapsto A_{f,\bar{x}}(t)$, $t\in \mathbb{R}$ of integrable function $f:\mathbb{R}^n\rightarrow \mathbb{R}$ by

\begin{eqnarray}\label{section function}
A_{f,\bar{x}}(t)=\int_{\{z\in\mathbb{R}^n:z\cdot \bar{x}=t\}}f(z)dz.
\end{eqnarray}
If $F= r \bar{x}$ (which belongs to $G_{n,1}$), $r\in \mathbb{R}$, then $F^{\bot}=\{y\in \mathbb{R}^n: y\cdot \bar{x}=0\}$ and
\begin{eqnarray*}\label{}
r_0 \bar{x}+F^{\bot}&=&\{r_0 \bar{x}+y\in \mathbb{R}^n: y\cdot \bar{x}=0\}\\
&=&\{z\in \mathbb{R}^n: z\cdot \bar{x}=r_0\},
\end{eqnarray*}
where $r_0 \bar{x}\in F$. Hence,
\begin{eqnarray}\label{}
A_{f,\bar{x}}(t)=\pi_F(f)(t).
\end{eqnarray}

Let $K\in \mathcal{S}_o^n$ and $u\in S^{n-1}$. The \emph{parallel section function} $t\mapsto A_{K,u}(t)$, $t\in \mathbb{R}$ of $K$ is defined by
\begin{eqnarray*}
A_{K,u}(t)=V_{n-1}(K\cap(u^{\perp}+tu)).
\end{eqnarray*}
Indeed,  $A_{f,\bar{x}}(t)$ can be called the \emph{functional version of parallel section function},
since it inherits almost all properties of the parallel section function of a star body.

A convex body  $K\subset\mathbb{R}^n$ is  $k$-smooth for some $k\in
\mathbb{N}$, if its boundary  is a $C^k$-smooth hypersurface of
$\mathbb{R}^n$. Let $C^k(\mathbb{R}^n)$ denote the set of $k$-times
continuously differentiable functions in $\mathbb{R}^n$.

\begin{lemma}\label{k times continuous differentiable}
Let $f\in\mathcal{A}_e\cap C^k(\mathbb{R}^n)$, where $k=0,1,2,\cdots$.  Then for all
 $\bar{x}=\tfrac{x}{\|x\|} \in  S^{n-1}$  ($x\in \mathbb{R}^n\backslash\{0\}$),  functions $A_{f,\bar{x}}(t)$ are uniformly  $k$-times continuous differentiable in some neighborhood of zero.
\end{lemma}

\begin{proof}
By layer-cake representation   and Fubini's theorem yield
\begin{eqnarray}\label{k times continuous differentiable1}
A_{f,\bar{x}}(t)&=&\int_{\{z\in\mathbb{R}^n:z\cdot \bar{x}=t\}}\int_0^{\infty}\mathcal{X}_{[f]_s}(z)dsdz \nonumber \\
&=&\int_0^{\infty}\int_{\{z\in [f]_s\cap\{y\in\mathbb{R}^n:y\cdot \bar{x}=t\}\}}dzds\nonumber\\
&=&\int_0^{\infty}A_{[f]_s,\bar{x}}(t)ds,
\end{eqnarray}
where
$$
[f]_s=\{x\in \mathbb{R}^n: f(x)\geq s\}
$$
and
$$
A_{[f]_s,\bar{x}}(t)=V_{n-1}( [f]_s\cap\{y\in\mathbb{R}^n:y\cdot \bar{x}=t\}).
$$
Since $f\in \mathcal{A}_e\cap C^k(\mathbb{R}^n)$,
 $$
 \partial [f]_s=\{x\in \mathbb{R}^n: f(x)= s\}
 $$
is a $C^k$ submanifold of $\mathbb{R}^n$ with non-zero  normal
vector $\nabla f$ and $f$ is integrable.  Therefore, $[f]_s$ is a $k$-smooth
origin-symmetric convex body in $\mathbb{R}^n$ for each $s>0$. By
Koldobsky's \cite[Lemma 2.4]{Koldobsky3}, we proved that $A_{[f]_s,
\bar{x}}(t)$ (uniformly with respect to $\bar{x}$) are $k$
continuous differentiable  in some neighborhood of zero. By (\ref{k
times continuous differentiable1}) we complete the proof.
\end{proof}

\begin{pro}
Let $f\in \mathcal{A}$. Fix $x\in \mathbb{R}^n\backslash\{0\}$, let
$\bar{x}=\frac{x}{\|x \|}$ and $A_{f,\bar{x}}(t)$ be defined in
(\ref{section function}). Then we have:

\begin{enumerate}
\item The function $A_{f,\bar{x}}(t)$ is log-concave on its support.

\item We have $\int_{\mathbb{R}}A_{f,\bar{x}}(t) dt=J(f)$.
\item If $f$ is even, then, for each $x\in\mathbb{R}^n\backslash\{0\}$ and $t\in \mathbb{R}$,
$$
A_{f,\bar{x}}(t)\leq A_{f,\bar{x}}(0).
$$
\item If $f\in\mathcal{A}_e\cap C^2(\mathbb{R}^n)$, then  $A_{f,\bar{x}}^{\prime\prime}(0)\leq 0$.
\end{enumerate}
\end{pro}
\begin{proof}

Note that (1) and (2) follow from Lemma \ref{BGVV's Proposition} and
(\ref{section function}). (3) comes from that a log-concave even
function has maximum at zero. If, in addition, $f$ is $2$-smooth,
then, as we showed  in Lemma \ref{k times continuous differentiable},
the function $A_{f,\bar{x}}(t)$ is twice differentiable in a
neighborhood of zero.  We proved  (4).
\end{proof}

\vskip 0.3cm~~
\subsection{Intersection functions}
~~\vskip 0.3cm

Let $L\in \mathcal{S}_o^n$. We recall that the intersection body, ${\rm I}L$, of star body $L$ is defined by
\begin{eqnarray*}
\rho_{{\rm I}L}(u)=V_{n-1}(L\cap u^{\perp}), \  u\in S^{n-1},
\end{eqnarray*}
where $u^{\perp}$ is the $(n-1)$-dimensional subspace orthogonal to the unit vector $u$. The radial function of ${\rm I}L$ equals the spherical Radon transform of $\frac{1}{n-1}\rho_L^{n-1}$, that is
\begin{eqnarray*}
\rho_{{\rm I}L}(u)={\rm R}\left(\frac{1}{n-1}\rho_L^{n-1}\right)(u).
\end{eqnarray*}
A slightly more general notion was defined  in \cite{GLW}, as follows. An origin-symmetric star body $K$ in $\mathbb{R}^n$ is said to be an intersection body if there exists a finite non-negative Borel measure $\mu$ on $S^{n-1}$ so that the radial function $\rho_K$ of $K$ equals the sphere Radon transform of $\mu$.

Inspired by the definition of intersection bodies, we define the intersection function by using marginals.

\begin{defi}
Let  $f: \mathbb{R}^n\rightarrow [0,\infty)$ be an integrable function. The intersection function, $\mathcal{I} f: \mathbb{R}^n\rightarrow [0,\infty)$,  of $f$ is defined as
\begin{eqnarray}\label{intersection-function}
\mathcal{I} f (x)=e^{-\|x\|A_{f,\bar{x}}(0)^{-1}},
\end{eqnarray}
when $x\in \mathbb{R}^n\setminus\{0\}$ with $\bar{x}=\frac{x}{\|x\|}$, and $\mathcal{I} f (0)=1$. In addition, we say that $g: \mathbb{R}^n\rightarrow [0,\infty)$ with $g(0)>0$ is an intersection function if there exists a non-negative  integrable function $f$ such that $g(x)=\mathcal{I} f (x)$.

\end{defi}

If $f: \mathbb{R}^n\to\mathbb{R}$ is a compactly supported continuous function, then  the definition of Rodan transform (\ref{radon transform}) yields
\begin{eqnarray*}
A_{f,\bar{x}}(0)&=&\int_{\{z\in\mathbb{R}^n:z\cdot \bar{x}=0\}}f(z)dz\\
&=&\|x\|^{}\mathcal{R}f(x,0).
\end{eqnarray*}
Hence, the intersection function can be rewritten as
\begin{eqnarray}
\mathcal{I} f (x)=e^{-\mathcal{R}f(x,0)^{-1}}.
\end{eqnarray}

The next lemma shows that the intersection function includes the intersection body.

\begin{lemma}\label{intersection body and function}
Let $K\in\mathcal{S}_o^n$. If $f(x)=e^{-c\|x\|_K}$, $x\in \mathbb{R}^n$,  with $c>0$, then
\begin{eqnarray}
\mathcal{I} f (x)=e^{-c^{(n-1)}\Gamma(n)^{-1}\|x\|_{{\rm I}K}}.
\end{eqnarray}
\end{lemma}
\begin{proof}
Trivially,   the equality in this lemma holds when $x=0$.
For $x\in\mathbb{R}^n\setminus\{0\}$, a direct calculation shows that
\begin{eqnarray*}
A_{f,\bar{x}}(0)
&=&\int_{\{z\in\mathbb{R}^n:z\cdot \bar{x}=0\}}f(z)dz\\
&=&\int_{0}^{\infty}\int_{S^{n-1}\cap \bar{x}^{\perp}}t^{n-2}e^{-tc\|v\|_K}dv dt\\
&=&c^{-(n-1)}\Gamma(n)\rho_{{\rm I}K}(\bar{x}).
\end{eqnarray*}
\end{proof}

Let $f: \mathbb{R}^n\rightarrow[0,\infty)$ be an integrable function with  $f(0)>0$. For any $p>0$, the set $K_p(f)$  was introduced by Ball \cite{Ball1988},
\begin{eqnarray}
K_p(f)=\left\{x\in \mathbb{R}^n: \int_0^{\infty}f(rx)r^{p-1}dr\geq\frac{f(0)}{p}\right\}.
\end{eqnarray}
From the definition it follows that the radial function of $K_p(f)$ is given by
\begin{eqnarray*}
\rho_{K_p(f)}(x)=\left(\frac{1}{f(0)}\int_0^{\infty}p r^{p-1}f(rx)dr\right)^{\frac{1}{p}}
\end{eqnarray*}
for $x\neq 0$.

The following properties were showed in \cite{BGVV}.
\begin{lemma}\label{BGVV's proposition 2.5.3}
Let $f: \mathbb{R}^n\rightarrow[0,\infty)$ be an integrable function with $f(0)>0$.
For every $p>0$, $K_p(f)$ has the following properties:
\begin{enumerate}

\item  $0\in K_p(f)$.
\item $K_p(f)$ is a star-shape set.

\item $K_p(f)$ is symmetric if $f$ is even.
\item $K_p(f)$ is a convex body if $f$ is log-concave and has finite positive integral.
\item $V(K_n(f))=\frac{1}{f(0)}J(f)$.

\end{enumerate}
\end{lemma}

A connection of intersection functions and intersection bodies is presented in the following lemma.

\begin{lemma}\label{intersection function and body}
Let $f: \mathbb{R}^n\rightarrow[0,\infty)$ be a continuous integrable function with $f(0)>0$. If $f$ has finite positive integral,  then
\begin{eqnarray}
\mathcal{I} f (x)=\exp\left\{-\frac{1}{f(0)}\|x\|_{{\rm I} K_{n-1}(f)}\right\},
\end{eqnarray}
for $x\in \mathbb{R}^n$.
\end{lemma}
\begin{proof}
It is trivial for $x=0$.
For $x\in \mathbb{R}^n\setminus\{0\}$, by Fubini's Theorem we have
\begin{eqnarray*}
A_{f,\bar{x}}(0)
&=&\int_{\{y\in \mathbb{R}^n:\bar{x}\cdot y=0\}}f(y)dy\\
&=&\int_0^{\infty}r^{n-2}\int_{S^{n-1}\cap \bar{x}^{\bot}}f(ru)dudr\\
&=&\int_{S^{n-1}\cap \bar{x}^{\bot}}\left(\int_0^{\infty} r^{n-2}f(ru)dr\right)du\\
&=&f(0){\rm R}\left(\frac{1}{n-1}\rho_{K_{n-1}(f)}^{n-1}\right)(\bar{x})\\
&=&f(0)\rho_{{\rm I} K_{n-1}(f)}(\bar{x}).
\end{eqnarray*}
The desired formula follows from the definition of intersection functions and (\ref{positively homogeneous}).
\end{proof}

Lutwak \cite{lutwak1991} proved  that for all $T\in {\rm GL}(n)$ and
all $K \in \mathcal{S}_o^n$,
\begin{eqnarray}\label{lutwak's affine}
{\rm I}(TK)=|\det T|T^{-t} {\rm I}K.
\end{eqnarray}
In  analytic case we have:

\begin{lemma}
Let  $f: \mathbb{R}^n\rightarrow [0,\infty)$ be a continuous integrable function with $f(0)>0$ and let  $T\in {\rm GL}(n)$.
Then

\begin{eqnarray}
\mathcal{I} (f\circ T) =(\mathcal{I} f)\circ \tfrac{T^{-t}}{|\det T^{-1}|}.
\end{eqnarray}
\end{lemma}

\begin{proof}
Let $T\in {\rm GL}(n)$. The definition of  $K_{p}(f)$ (where $p>0$) tells us that
\begin{eqnarray*}
\rho_{K_p(f\circ T)}(x)&=&\left(\frac{1}{f(0)}\int_0^{\infty}p r^{p-1}f(rT x)dr\right)^{\frac{1}{p}}\\
&=&\rho_{K_p(f)}(Tx)\\
&=&\rho_{T^{-1}K_p(f)}(x).
\end{eqnarray*}
Combining with Lutwak's result (\ref{lutwak's affine}) and Lemma \ref{intersection function and body}, we have

\begin{eqnarray*}
\mathcal{I} (f\circ T) (x)&=&\exp\left\{-\frac{1}{f(0)}\|x\|_{{\rm I}(T^{-1} K_{n-1}(f))}\right\}\\
&=&\exp\left\{-\frac{1}{f(0)}\|x\|_{|\det T|^{-1}T^{t} {\rm I} K_{n-1}(f)}\right\}\\
&=&\exp\left\{-\frac{1}{f(0)}\| |\det T|^{}T^{-t}x\|_{ {\rm I} K_{n-1}(f)}\right\}\\
&=&\mathcal{I} f \left(\tfrac{T^{-t}}{|\det T^{-1}|}x\right).
\end{eqnarray*}
\end{proof}

For an integrable function $f:\mathbb{R}^n\to\mathbb{R}$, the dual difference function $\widetilde{D}f$,  of $f$ is defined as

\begin{eqnarray}\label{dual-difference-function}
\widetilde{D}f(x)=\tfrac{1}{2}f(x)+\tfrac{1}{2}f(-x), \quad x\in\mathbb{R}^n.
\end{eqnarray}

\begin{lemma}\label{difference}
Let $\varphi:\mathbb{R}^n\rightarrow \mathbb{R}$ be a continuous function such that $f=e^{-\varphi}$ is integrable. Then
\begin{eqnarray}\label{}
\mathcal{I}(\widetilde{D}f)=\mathcal{I}f,
\end{eqnarray}
and
\begin{eqnarray}\label{}
\widehat{\delta J}(f,f)\leq\widehat{\delta J}(\widetilde{D}f,\widetilde{D}f)
\end{eqnarray}
with equality if and only if $f$ is even.
\end{lemma}
\begin{proof}
 Since
\begin{eqnarray*}
A_{\widetilde{D}f,\bar{x}}(0)&=&\frac{1}{2} \int_{\{z\in\mathbb{R}^n:z\cdot \bar{x}=0\}}(f(z)+f(-z))dz\\
&=&A_{f,\bar{x}}(0),
\end{eqnarray*}
hence, from the definition of intersection functions we deduce  that $\mathcal{I}(\widetilde{D}f)=\mathcal{I}f$.

Let $\bar{f}(x)=f(-x)$. Then $J(f)=J(\bar{f})$, and from Theorem  \ref{functional dual Minkowski inequality0} we have
\begin{eqnarray*}\label{}
\widehat{\delta J}(\widetilde{D}f,\widetilde{D}f)
&=&-\int_{\mathbb{R}^n}\widetilde{D}f (x)\log\widetilde{D}f(x) dx\\
&=&\frac{1}{2}\widehat{\delta J}(f,\widetilde{D}f)+\frac{1}{2}\widehat{\delta J}(\bar{f},\widetilde{D}f)\\
&\geq&\frac{1}{2}\widehat{\delta J}(f,f)+\frac{1}{2}\widehat{\delta J}(\bar{f},\bar{f})\\
&=&\widehat{\delta J}(f,f),
\end{eqnarray*}
with equality if and only if $f=\bar{f}$, namely, $f$ is  even.
\end{proof}

\vskip 0.3cm ~~
\subsection{The Busemann intersection inequality for functions}
~~\vskip 0.3cm

For  $K\in \mathcal{S}_o^n$, the Busemann intersection inequality says that (see, e.g., \cite{schneider})
\begin{eqnarray}\label{Busemann intersection inequality}
V({\rm I} K)\leq\frac{\omega_{n-1}^{n}}{\omega_{n}^{n-2}}V(K)^{n-1},
\end{eqnarray}
with equality for $n=2$ if and only if $K$ is an origin-symmetric star body, and for $n\geq 3$ if and only if $K$ is  an  origin-symmetric ellipsoid.

In this subsection, we will prove the Busemann intersection inequality for functions.  A proof of the following lemma  can be found   in \cite{BGVV} when the function $f$ is log-concave, and we note that the log-concavity can be removed. We remark that the equality  condition is a new result.

\begin{lemma}\label{increasing}
\emph{(\cite[Lemma 2.2.4]{BGVV}).} Let $f:[0,\infty)\rightarrow[0,\infty)$ be a bounded integrable function. If $0<p< q<\infty$, then
\begin{eqnarray}\label{increasing0}
\left(\frac{p}{\|f\|_{\infty}}\int_0^{\infty}r^{p-1}f(r)dr\right)^{\frac{1}{p}}\leq \left(\frac{q}{\|f\|_{\infty}}\int_0^{\infty}r^{q-1}f(r)dr\right)^{\frac{1}{q}}.
\end{eqnarray}
Moreover,
\begin{eqnarray}\label{}
\left(\frac{p}{\|f\|_{\infty}}\int_0^{\infty}r^{p-1}f(r)dr\right)^{\frac{1}{p}}= \left(\frac{q}{\|f\|_{\infty}}\int_0^{\infty}r^{q-1}f(r)dr\right)^{\frac{1}{q}}=c,
\end{eqnarray}
if and only if  $f(r)=\|f\|_{\infty}$ when $r\in[0,c]$ and $f(r)=0$ when $r\in(c,\infty)$.
\end{lemma}

\begin{proof}
Without loss of generality we may assume that $\|f\|_{\infty}=1$.  We set
\begin{eqnarray*}
F(p)=\left(p\int_0^{\infty}r^{p-1}f(r)dr\right)^{\frac{1}{p}}.
\end{eqnarray*}
Then for any $0<p<q$ and $\alpha>0$, we have
\begin{eqnarray}\label{increasing1}
\frac{F(q)^q}{q}&=&\int_0^{\infty}r^{q-1}f(r)dr \nonumber \\
&=&\int_0^{\alpha}r^{q-1}f(r)dr+\int_{\alpha}^{\infty}r^{q-1}f(r)dr \nonumber\\
&\geq&\int_0^{\alpha}r^{q-1}f(r)dr+\alpha^{q-p}\int_{\alpha}^{\infty}r^{p-1}f(r)dr.
\end{eqnarray}
We observe  that equality holds if and only if $f(r)=0$ when $r\in(\alpha,\infty)$. Moreover,
\begin{eqnarray}\label{increasing2}
&&\int_0^{\alpha}r^{q-1}f(r)dr+\alpha^{q-p}\int_{\alpha}^{\infty}r^{p-1}f(r)dr \nonumber\\
&& \quad\quad = \  \alpha^{q-p}\frac{F(p)^p}{p}-\alpha^q\int_0^{1}(r^{p-1}-r^{q-1})f(\alpha r)dr \nonumber\\
&&\quad\quad  \geq \  \alpha^{q-p}\frac{F(p)^p}{p}-\alpha^q\left(\frac{1}{p}-\frac{1}{q}\right).
\end{eqnarray}
Equality holds if and only if $f(\alpha r)=1$ when $r\in[0,1]$ and $\alpha>0$. The desired inequality follows from the choice $\alpha=F(p)$. The equality conditions of (\ref{increasing1}) and (\ref{increasing2}) imply that equality holds in (\ref{increasing0}) if and only if  $f(r)=1$ when $r\in[0,\alpha]$ and $f(r)=0$ when $r\in(\alpha,\infty)$ with $\alpha>0$.
\end{proof}

The following inclusion has been showed in \cite{BGVV} for log-concave functions.
We give a slightly more generalized  version as follow.

\begin{lemma}\label{inclusion}
 Let $f:\mathbb{R}^n\rightarrow[0,\infty)$ be a bounded, continuous  integrable function with $f(0)>0$. If $0<p\leq q$, then
 \begin{eqnarray}\label{inclusion1}
\left(\frac{\|f\|_{\infty}}{f(0)}\right)^{-\frac{1}{p}}K_p(f)\subseteq
\left(\frac{\|f\|_{\infty}}{f(0)}\right)^{-\frac{1}{q}}K_q(f).
\end{eqnarray}
Moreover,

\begin{eqnarray}\label{}
\left(\frac{\|f\|_{\infty}}{f(0)}\right)^{-\frac{1}{p}}K_p(f)=\left(\frac{\|f\|_{\infty}}{f(0)}\right)^{-\frac{1}{q}}K_q(f)=K
\end{eqnarray}
if and only if  $f$ is a constant on $K$ and vanishes outside of $K$.
\end{lemma}

\begin{proof}
Without loss of generality we may assume that $\|f\|_{\infty}=1$. For any $u\in S^{n-1}$, by Lemma \ref{increasing} we have
\begin{eqnarray*}
\rho_{K_q(f)}(u)&=&\left(\frac{1}{f(0)}\int_0^{\infty}q r^{q-1}f(ru)dr\right)^{\frac{1}{q}}\\
&=&\left(\frac{1}{f(0)}\right)^{\frac{1}{q}}\left(\int_0^{\infty} qr^{q-1}f(ru)dr\right)^{\frac{1}{q}}\\
&\geq&\left(\frac{1}{f(0)}\right)^{\frac{1}{q}}\left(\int_0^{\infty} pr^{p-1}f(ru)dr\right)^{\frac{1}{p}}\\
&=&\left(\frac{1}{f(0)}\right)^{\frac{1}{q}-\frac{1}{p}}\rho_{K_p(f)}(u).
\end{eqnarray*}
The equality condition of Lemma \ref{increasing} implies that there is an equality in (\ref{inclusion1}) if and only if $f$ is a positive  constant on $K$ and vanishes  outside of $K$.
\end{proof}

\begin{thm}\label{functional intersection inequality}
 Let $f:\mathbb{R}^n\rightarrow[0,\infty)$ be a bounded, continuous  integrable function with $f(0)>0$. Then
\begin{eqnarray}\label{functional intersection inequality1}
J(\mathcal{I} f)\leq\Gamma(n+1)\frac{\omega_{n-1}^{n}}{\omega_{n}^{n-2}}
\|f\|_{\infty}^{}J(f)^{n-1},
\end{eqnarray}
with equality  for $n=2$ if and only if $cf$ is  the characteristic function of  an  origin-symmetric star  body, and for $n\geq 3$ if and only if $cf$ is  the characteristic function of  an  origin-symmetric ellipsoid, where $c>0$.
\end{thm}

\begin{proof}
From Lemma \ref{intersection function and body} and the Busemann intersection inequality (\ref{Busemann intersection inequality}), we have
\begin{eqnarray*}\label{}
J(\mathcal{I} f)&=&\int_{\mathbb{R}^n}e^{-\frac{1}{f(0)}\|x\|_{{\rm I} K_{n-1}(f)}}dx\\
&=&\Gamma(n+1)f^n(0)V({\rm I} K_{n-1}(f))\\
&\leq&\Gamma(n+1)f^n(0)\frac{\omega_{n-1}^{n}}{\omega_{n}^{n-1}}V(K_{n-1}(f))^{n-1},
\end{eqnarray*}
with equality for $n=2$ if and only if $K_{n-1}(f)$ is an origin-symmetric convex body, and for $n\geq 3$ if and only if $K_{n-1}(f)$ is  an  origin-symmetric ellipsoid.

From (\ref{inclusion1}) and Lemma \ref{BGVV's proposition 2.5.3}, it follows that
\begin{eqnarray*}\label{}
J(\mathcal{I} f)
&\leq&\Gamma(n+1)f^n(0)\frac{\omega_{n-1}^{n}}{\omega_{n}^{n-1}}V(K_{n-1}(f))^{n-1}\\
&\leq&\Gamma(n+1)f^n(0)\frac{\omega_{n-1}^{n}}{\omega_{n}^{n-1}}
V\left(\left(\frac{\|f\|_{\infty}}{f(0)}\right)^{\frac{1}{n-1}-\frac{1}{n}}K_n(f)\right)^{n-1}\\
&=&\Gamma(n+1)f^n(0)\frac{\omega_{n-1}^{n}}{\omega_{n}^{n-1}}
\left(\frac{\|f\|_{\infty}}{f(0)}\right)^{}
\left(\frac{J(f)}{f(0)}\right)^{n-1}.
\end{eqnarray*}
The desired equality conditions come from the equality conditions of Lemma \ref{inclusion} and the Busemann intersection inequality (\ref{Busemann intersection inequality}).
\end{proof}

\section{The solutions to Problems  \ref{integral case} and   \ref{dual case} }\label{FBP}

\vskip 0.3cm

In this section we will solve Problems  \ref{integral case} and   \ref{dual case}.  It is not hard to see that the even condition in Problem  \ref{integral case} is necessary, since it  recovers  the geometric case when functions $f$ and $g$ are limited to  the characteristic function of  convex bodies. Next, we prove that  the even condition in Problem \ref{dual case} is also necessary.

\begin{lemma}\label{not even}
Let $f, g$ be   integrable log-concave functions in $\mathbb{R}^n$. If $f$ is not even, then there exists
an even continuous function $g$ such that for $x\in\mathbb{R}^n$
$$
 \mathcal{I}g (x)<\mathcal{I}f (x)
$$
with
$$
\widehat{\delta J}(g,g)>\widehat{\delta J}(f,f).
$$
\end{lemma}

\begin{proof}
Lemma \ref{difference} tells that the intersection function of $\widetilde{D}f$ agree with that of $f$, that is,
\begin{eqnarray*}\label{}
A_{\widetilde{D}f,\bar{x}}(0)=A_{f,\bar{x}}(0),\quad  x\in \mathbb{R}^n\setminus\{0\}.
\end{eqnarray*}
Since $f$ is not even,  by
Lemma \ref{difference} we have

\begin{eqnarray}\label{not even1}
\widehat{\delta J}(f,f)<\widehat{\delta J}(\widetilde{D}f,\widetilde{D}f).
\end{eqnarray}
Let $f=e^{-\varphi}$ and $\bar{f}=e^{-(\varphi+c)}$ with $c\in \mathbb{R}$.
It is clear that
$\widetilde{D}\bar{f}=e^{-c}\widetilde{D}f$, and
\begin{eqnarray*}\label{}
\widehat{\delta J}(\bar{f},\bar{f})&=&e^{-c}\left[\widehat{\delta J}(f,f)+cJ(f)\right]\\
&<&e^{-c}\left[\widehat{\delta J}(\widetilde{D}f,\widetilde{D}f)+cJ(\widetilde{D}f)\right]\\
&=&\widehat{\delta J}(\widetilde{D}\bar{f},\widetilde{D}\bar{f}).
\end{eqnarray*}
This means that adding a constant to $\varphi$ does not change inequality (\ref{not even1}). We may therefore add a constant  to $\varphi$ and assume that
$$
\inf\varphi>1.
$$
Combining this assumption with Theorem \ref{integral formula}, we infer that  $\widehat{\delta J}(f,f)>0$.

Let $g(x)=\varepsilon \widetilde{D}f$ with
$$
\varepsilon=\frac{1}{2}\left(1+\frac{\widehat{\delta J}(f,f)}{\widehat{\delta J}(\widetilde{D}f,\widetilde{D}f)}\right)<1.
$$
 Then, we have
\begin{eqnarray*}\label{}
A_{g,\bar{x}} (0)=\varepsilon A_{\widetilde{D}f,\bar{x}} (0)>A_{f,\bar{x}} (0),
\quad {\text for}\quad  x\in \mathbb{R}^n\setminus\{0\}
\end{eqnarray*}
and
\begin{eqnarray*}\label{}
\widehat{\delta J}(g,g)
&=&\varepsilon\left[\widehat{\delta J}(\widetilde{D}f,\widetilde{D}f)-\log\varepsilon J(\widetilde{D}f)\right]\\
&>&\varepsilon\widehat{\delta J}(\widetilde{D}f,\widetilde{D}f)\\
&=&\frac{1}{2}\widehat{\delta J}(\widetilde{D}f,\widetilde{D}f)+\frac{1}{2}\widehat{\delta J}(f,f)\\
&>&\widehat{\delta J}(f,f).
\end{eqnarray*}
Therefore $g$ is the even continuous function we are looking for.
\end{proof}

To prove the characteristic theorem  we need the following two  geometric results.

\begin{lemma}\label{Lutwak's theorem 10.1}
\emph{(\cite[Theorem 10.1]{lutwak1988})}.
If $K$ is an intersection body and $L$ is an origin-symmetric star body, and for all $u\in S^{n-1}$,
\begin{eqnarray*}\label{}
V_{n-1}(K\cap u^{\perp})\leq V_{n-1}(L\cap u^{\perp}),
\end{eqnarray*}
then
\begin{eqnarray*}\label{}
V(K)\leq V(L),
\end{eqnarray*}
with equality if and only if $K=L$.
\end{lemma}

\begin{lemma}\label{Lutwak's theorem 12.2}
\emph{(\cite[Theorem 12.2]{lutwak1988})}.
If $K$ is an origin-symmetric  star body whose radial function is in $C^{\infty}_e(S^{n-1})$, then if $K$ is not an intersection body, there exists an origin-symmetric body $L$, such that, for all $u\in S^{n-1}$
\begin{eqnarray*}\label{}
V_{n-1}(L\cap u^{\perp})<V_{n-1}(K\cap u^{\perp}),
\end{eqnarray*}
but
\begin{eqnarray*}\label{}
V(L)>V(K).
\end{eqnarray*}
\end{lemma}

We are ready to prove the characteristic theorem presented in Introduction.

\begin{pro}\label{characteristic theorem}
Problems \ref{integral case} and   \ref{dual case} have   affirmative answers  in $\mathbb{R}^n$ if and only if every even integrable log-concave function in $\mathbb{R}^n$ is an intersection function.
\end{pro}

\begin{proof}

The  proofs of  the statements for   Problems \ref{integral case}  and   \ref{dual case} in this proposition  are similar, hence,  we only give a  detailed proof for  Problem \ref{dual case} and the case of Problem \ref{integral case}   follows from  the same line.

Assume that both $f$ and $g$ are even integrable log-concave functions  and  intersection functions.
Let $f_0, \ g_0$ be  integrable continuous  functions  with $f_0(0)>0$ and $g_0(0)>0$  such that
$\mathcal{I}f_0 (x)=f(x)$ and $\mathcal{I}g_0 (x)=g(x)$ for $x\in\mathbb{R}^n$.
To prove  the  sufficiency, it suffices to show  that  for any $x\in\mathbb{R}^n$
\begin{eqnarray*}\label{}
\mathcal{I}f (x)\leq \mathcal{I}g (x)
\end{eqnarray*}
implies that
\begin{eqnarray*}\label{}
\widehat{\delta J}(f,f)
\leq\widehat{\delta J}(g,g),
\end{eqnarray*}
with equality if and only if $f=g$.

Lemma \ref{intersection function and body} implies that for  any $x\in\mathbb{R}^n$
\begin{eqnarray*}\label{}
f(x)=\exp\left\{-\frac{1}{f_0(0)}\|x\|_{{\rm I} K_{n-1}(f_0)}\right\},
\end{eqnarray*}
and
\begin{eqnarray*}\label{}
g(x)=\exp\left\{-\frac{1}{g_0(0)}\|x\|_{{\rm I} K_{n-1}(g_0)}\right\}.
\end{eqnarray*}
By Lemma \ref{intersection body and function},  the assumption
$$
\mathcal{I}f (x)\leq \mathcal{I}g (x)
$$
is equivalent to
\begin{eqnarray*}\label{}
f_0(0)^{n-1}{\rm I}({\rm I} K_{n-1}(f_0))\subseteq g_0(0)^{n-1}{\rm I}({\rm I} K_{n-1}(g_0)).
\end{eqnarray*}
By Lemma \ref{Lutwak's theorem 10.1}
we have
\begin{eqnarray}\label{characteristic theorem1}
f_0(0)^{n}V({\rm I} K_{n-1}(f_0))\leq g_0(0)^{n}V({\rm I} K_{n-1}(g_0)),
\end{eqnarray}
with equality if and only if ${\rm I} K_{n-1}(f_0)={\rm I} K_{n-1}(g_0)$, i.e., $f=g$.  On the other hand, by Lemma \ref{includes geometry case}  and (\ref{characteristic theorem1}) we have
\begin{eqnarray*}\label{}
\widehat{\delta J}(f,f)&=&n\Gamma(n+1)f_0(0)^nV({\rm I} K_{n-1}(f_0))\\
&\leq&n\Gamma(n+1)g_0(0)^nV({\rm I} K_{n-1}(g_0))\\
&=&\widehat{\delta J}(g,g),
\end{eqnarray*}
with equality if and only if $f=g$. This finishes the proof of sufficiency.

To prove the necessity part, it suffices to show that the existence of nonintersection
functions implies a negative answer to Problem \ref{dual case}.

 Let $K$ be an origin-symmetric  star body whose radial function is in $C^{\infty}_e(S^{n-1})$, and let $K$ be not an intersection body. We claim that $f(x)=e^{-\|x\|_K}\in C_e^{\infty}(\mathbb{R}^n)$ and it is not an intersection function.   In fact, if $e^{-\|x\|_K}$ is an intersection function, then there exists a non-negative integrable function $h$ in $\mathbb{R}^n$ such that $f(x)=e^{-\|x\|_K}=e^{-(\mathcal{R}h)^{-1}(x,0)}$. By Lemma \ref{intersection function and body}, we have
\begin{eqnarray*}
e^{-\|x\|_K}=\mathcal{I} h (x)=\exp\{-h(0)\rho_{{\rm I} K_{n-1}(h)}^{-1}(x)\},
\end{eqnarray*}
which  implies that $K$ is an intersection body. This is a contradiction.

By a direct calculation, we know  that if $f(x)=e^{-\|x\|_K}$ then
\begin{eqnarray*}\label{}
\widehat{\delta J}(f,f)=n\Gamma(n+1)V(K).
\end{eqnarray*}
From Lemma \ref{Lutwak's theorem 12.2}, it follows that there exists an even function $g(x)=e^{-\|x\|_L}$ (where $L$ is an origin-symmetric star body), such that
\begin{eqnarray*}\label{}
\mathcal{I}f (x)< \mathcal{I}g (x)
\end{eqnarray*}
but
\begin{eqnarray*}\label{}
\widehat{\delta J}(f,f)
>\widehat{\delta J}(g,g).
\end{eqnarray*}
We complete the proof.
\end{proof}

Next, we investigate the equivalence between intersection functions and intersection bodies.

\begin{lemma}\label{relation of intersection function and body}
Let $f:\mathbb{R}^n\to\mathbb{R}$ be a non-negative  continuous integrable function with $f(0)>0$. Then $f$ is an intersection function if and only if  $K_{n-1}(f)$ is an intersection body.
\end{lemma}

\begin{proof}
Let $g:\mathbb{R}^n\to\mathbb{R}$ be a non-negative, continuous,  integrable function with $g(0)>0$.
Assume that $f$ is an intersection function of $g$. By Lemma \ref{intersection function and body}, we have
\begin{eqnarray}\label{relation of intersection function and body1}
f(x)=\mathcal{I} g (x)=\exp\left\{-\frac{1}{g(0)}\|x\|_{{\rm I} K_{n-1}(g)}\right\}.
\end{eqnarray}
Moreover,  from the definition of the body $K_{n-1}(f)$  for $x\in \mathbb{R}^n\setminus\{0\}$ we have
\begin{eqnarray*}
\rho_{K_{n-1}(f)}(x)&=&
\left(\frac{1}{f(0)}\int_0^{\infty}(n-1) r^{n-2}\exp\left\{-\frac{1}{g(0)}\|rx\|_{{\rm I} K_{n-1}(g)}\right\}dr\right)^{\frac{1}{n-1}}\\
&=&\Gamma(n)^{\frac{1}{n-1}}f(0)^{-\frac{1}{n-1}} g(0)\rho_{{\rm I} K_{n-1}(g)}(x).
\end{eqnarray*}
This means that $K_{n-1}(f)$ is an intersection body.

If $K_{n-1}(f)$ is an intersection body, there exists a star body $L$ in $\mathbb{R}^n$ such that
\begin{eqnarray*}
\rho_{K_{n-1}(f)}(x)=\rho_{{\rm I}L}(x).
\end{eqnarray*}
Let $g(x)=e^{-\|x\|_L}$. It is clear that
$$
K_{n-1}(g)=\Gamma(n)^{-\frac{1}{n-1}}L.
$$
 By Lemma \ref{intersection function and body}, we have
\begin{eqnarray*}
\rho_{K_{n-1}(f)}(x)&=&\Gamma(n)^{-1}\rho_{{\rm I}K_{n-1}(g)}(x)\\
&=&\Gamma(n)^{-1}\mathcal{R}g(x,0).
\end{eqnarray*}
A direct calculation shows that
\begin{eqnarray*}
\rho_{K_{n-1}(\mathcal{I} g )}^{n-1}(x)
&=&(n-1)\int_0^{\infty} r^{n-2} e^{-\mathcal{R}g(rx,0)^{-1}}dr\\
&=&(n-1)\int_0^{\infty} r^{n-2} e^{-r\mathcal{R}g(x,0)^{-1}}dr\\
&=&\Gamma(n)\mathcal{R}g(x,0)^{n-1}.
\end{eqnarray*}
Therefore,
\begin{eqnarray*}\label{}
\rho_{K_{n-1}(f)}(x)=\Gamma(n)^{-\frac{n}{n-1}}\rho_{K_{n-1}(\mathcal{I} g )}(x).
\end{eqnarray*}
Hence, $f$ is an intersection function.
\end{proof}

Now, we give a positive answer to Problem \ref{integral case} and  Problem \ref{dual case} when $2\leq n\leq 4$.

\begin{thm}\label{an answer to question3}
Problems \ref{integral case} and   \ref{dual case} have   affirmative answers when $2\leq n\leq 4$.
\end{thm}
\begin{proof}
In \cite{GKS}, authors proved that every origin-symmetric convex body in $\mathbb{R}^n$
is an intersection body when $2\leq n\leq 4$.
Therefore, By Proposition \ref{characteristic theorem} and Lemma \ref{relation of intersection function and body},
 Problems \ref{integral case} and   \ref{dual case} have   affirmative answers  when $2\leq n\leq 4$.
\end{proof}

\section{Proof of Theorem \ref{main theorem}}\label{Proof of Theorem}

In this section we will finish the proof of Theorem \ref{main theorem}.
The following result is the key reason why Problem \ref{entropy case}, Problem \ref{integral case} and Problem \ref{dual case} can be considered the functional Busemann-Petty problems.
\begin{pro}\label{negative answer}
Problem \ref{entropy case}, Problem \ref{integral case} and Problem \ref{dual case} all include the Busemann-Petty problem.
\end{pro}
\begin{proof}
Let $K$ be an origin-symmetric convex body in $\mathbb{R}^n$.  For $x\in \mathbb{R}^n$ and $a>0$, if $f(x)=e^{-\|x\|_{aK}}$
then a direct calculation shows that
\begin{eqnarray}\label{special case1}
\int_{\mathbb{R}^n\cap u^{\bot}}f(x)dx=a^{n-1}\Gamma(n)V_{n-1}(K\cap u^{\bot}),
\end{eqnarray}
and
\begin{eqnarray}\label{special case2}
J(f)=a^{n}\Gamma (n+1) V(K),
\end{eqnarray}
\begin{eqnarray}\label{special case3}
\widehat{\delta J}(f,f)&=&a^{n}n\Gamma (n+1)V(K).
\end{eqnarray}
Moreover,
\begin{eqnarray}\label{special case4}
{\rm Ent}(f)&=&-\widehat{\delta J}(f,f)-J(f)\log J(f) \nonumber \\
&=&-a^{n}\Gamma (n+1)V(K)\left[n+\log\Gamma (n+1)+\log( a^{n}V(K))\right].
\end{eqnarray}
Then, (\ref{special case1}) and (\ref{special case2}) show that Problem \ref{integral case} includes the Busemann-Petty problem, and (\ref{special case1}) and (\ref{special case3}) show that Problem \ref{dual case} includes the Busemann-Petty problem.

Let $K,L$ be  any origin-symmetric convex bodies in $\mathbb{R}^n$.
 There exists a constant $a>0$ (for example, $a^{-n}=\min\{V(K), V(L)\}$) such that
$$
V(aK)\geq 1\quad
\text{and}  \quad  V(aL)\geq 1.
$$
If $f(x)=e^{-\|x\|_{aK}}$ and  $g(x)=e^{-\|x\|_{aL}}$, $x\in \mathbb{R}^n$, then (\ref{special case1}) deduces that
\begin{eqnarray*}
\int_{\mathbb{R}^n\cap u^{\bot}}f(x)dx\leq \int_{\mathbb{R}^n\cap u^{\bot}}g(x)dx  \ \Longrightarrow \ V_{n-1}(K\cap u^{\bot})\leq V_{n-1}(L\cap u^{\bot})
\end{eqnarray*}
for every $u\in S^{n-1}$. By (\ref{special case4}) and  the facts $ V(a K)\geq 1$ and $ V(a L)\geq 1$ (i.e.,  $\log V(a K)\geq 0$ and $\log V(a L)\geq 0$)  
we have
\begin{eqnarray*}\label{}
{\rm Ent}(f)\geq{\rm Ent}(g) \ \Longrightarrow \  V(K)\leq V(L).
\end{eqnarray*}
This completes our proof.
\end{proof}

Proposition \ref{negative answer} tells us that that  Problem \ref{entropy case}, Problem \ref{integral case} and Problem \ref{dual case} have negative answers when $n\geq 5$.  Hence, we only need to prove Theorem \ref{main theorem} for $2\leq n\leq 4$.

\vskip  0.2cm

\begin{proof}[\textbf{Proof of Theorem \ref{main theorem}}]
 Problem \ref{integral case} and Problem \ref{dual case}
have been proved Theorem \ref{an answer to question3}, respectively.
We only need to solve  Problem \ref{entropy case}.

Assume that $J(f)>e^{-1}$ and $J(g)>e^{-1}$. If $2\leq n\leq 4$, then Problem \ref{integral case} and
Problem \ref{dual case} have positive answers. Namely,
for integrable  even  log-concave functions  $f,g: \mathbb{R}^n\to\mathbb{R}$   with   finite positive integrals, if  $2\leq n\leq 4$ and 
\begin{eqnarray*}
\int_{\mathbb{R}^n\cap H}f(x)dx\leq\int_{\mathbb{R}^n\cap H}g(x)dx
\end{eqnarray*}
for every hyperplane $H$ passing through the origin, then
\begin{eqnarray}\label{F}
 J(f)    \leq            J(g)\quad \text{and} \quad \widehat{\delta J}(f,f)\leq \widehat{\delta J}(g,g).
\end{eqnarray}
 Since the function $t\log t$ (where $t>e^{-1}$) is strictly increasing and (\ref{F}),
 hence
\begin{eqnarray*}
{\rm Ent}(f)&=&-\widehat{\delta J}(f,f)-J(f)\log J(f)\\
&\geq& -\widehat{\delta J}(g,g)-J(g)\log J(g)\\
&=& {\rm Ent}(g).
\end{eqnarray*}
Problem \ref{entropy case} has positive answer when $2\leq n\leq 4$. 

In general case when assumptions $J(f)>e^{-1}$ and $J(g)>e^{-1}$ are removed,  we consider functions $\bar{f}=cf$
and $\bar{g}=cg$
with $c=\max\{J(f)^{-1},J(g)^{-1}\}$.
Therefore $J(\bar{f})\geq1>e^{-1}$, $J(\bar{g})\geq1>e^{-1}$ and
\begin{eqnarray*}\label{}
{\rm Ent}(\bar{f})&=&\int_{\mathbb{R}^n}\bar{f}\log \bar{f}dx-J(\bar{f})\log J(\bar{f})\\
&=&c\int_{\mathbb{R}^n}f\log (cf)dx-cJ(f)\log (cJ(f))\\
&=&c{\rm Ent}(f).
\end{eqnarray*}
Since $c>0$, then
\begin{eqnarray*}
\int_{\mathbb{R}^n\cap H}f(x)dx\leq\int_{\mathbb{R}^n\cap H}g(x)dx
\end{eqnarray*}
is equivalent to
\begin{eqnarray*}
\int_{\mathbb{R}^n\cap H}\bar{f}(x)dx\leq\int_{\mathbb{R}^n\cap H}\bar{g}(x)dx.
\end{eqnarray*}
Hence we conclude that Problem \ref{entropy case} has  affirmative answer  for $2\leq n\leq 4$ 
 when the assumptions $J(f)>e^{-1}$ and $J(g)>e^{-1}$ are removed.
\end{proof}

After work on this project was completed, the authors learned of the work of Lv \cite{Lv}.  While there is some overlap of results, the methods employed to achieve them are quite different.

\vskip 0.3cm
~~

\vskip 1 cm
%

\end{document}